\documentclass{amsart}
\usepackage{amssymb,verbatim,enumerate,ifthen}
\usepackage{esint}
\usepackage[mathscr]{eucal}
\usepackage[utf8]{inputenc}
\usepackage[T1]{fontenc}
\usepackage{enumerate,ifthen}

\newcommand{\N}{\mathbb{N}}
\newcommand{\R}{\mathbb{R}}
\newtheorem{theorem}{Theorem}

\newtheorem{definition}[theorem]{Definition}

\newtheorem{lemma}[theorem]{Lemma}

\newtheorem{proposition}[theorem]{Proposition}
\newtheorem{remark}[theorem]{Remark}

\newtheorem{example}[theorem]{Example}
\numberwithin{theorem}{section}
\author{Janusz Matkowski}
\address{
Faculty of Mathematics, 
Computer Science and Econometrics,
University of Zielona G\'ora,
Szafrana 4a, PL-65-516 Zielona G\'ora,
Poland}
\email{j.matkowski@wmie.uz.zgora.pl}

\author{Pawe{\l} Pasteczka}
\address{
Institute of Mathematics,
Pedagogical University of Cracow,
Podchor\c{a}\.zych 2, PL-30-084 Krak\'ow, Poland}
\email{pawel.pasteczka@up.krakow.pl}


\title{Invariant means and iterates of mean-type mappings}
\dedicatory{Dedicated to Professor Janos Acz\'el on the occasion of his 95th birthday.}

\begin{document}
\begin{abstract}
Classical result states that for two continuous and strict means $M,\,N \colon I^2 \to I$ ($I$ is an interval) there exists a unique $(M,N)$-invariant mean $K \colon I^2 \to I$, i.e. such a mean that $K \circ (M,N)=K$ and, moreover, the sequence of iterates $((M,N)^n)_{n=1}^\infty$ converge to $(K,K)$ pointwise.

Recently it was proved that continuity assumption cannot be omitted in general. We show that if $K$ is a unique $(M,N)$-invariant mean then, under no continuity assumption, $(M,N)^n \to (K,K)$.
\end{abstract}

\maketitle

\section{Introduction}
It is known that \textit{if }$M$\textit{\ and }$N$\textit{\ are continuous
bivariable means in an interval }$I,$ \textit{and the mean-type mapping} $%
\left( M,N\right) $\textit{\ is diagonally contractive, that is, } 
\begin{equation}
\left\vert M\left( x,y\right) -N\left( x,y\right) \right\vert <\left\vert
x-y\right\vert ,\qquad\qquad x,y\in I,\text{ }x\neq y, 
\label{A}
\end{equation}%
\textit{then there is a unique mean }$K:I^{2}\rightarrow I$\textit{\ that is 
}$\left( M,N\right) $\textit{-invariant and, moreover the sequence of
iterates }$\left( \left( M,N\right) ^{n}\right) _{n\in \N}$\textit{\ of the
mean-type mapping }$\left( M,N\right) $\textit{\ converges \ to} $\left(
K,K\right) $ (\cite{JM2013}).

For the results of this type, with more restrictive assumptions (see \cite{BorBor}). In particular, instead of \eqref{A} it was assumed that both means are
strict; in \cite{JM99} it was assumed that at most one mean is not strict,
and condition \eqref{A} appeared first in \cite{JM-ECIT2009} (see also \cite{P.Pasteczka}). Moreover, in all these papers the uniqueness of the
invariant mean was obtained under the condition that it is continuous. In 
\cite{JM2013} it was shown that this condition is superfluous. In \cite{JMforAczel} it is shown that condition \eqref{A} holds if $\ M$ is right-strict
in both variables and $N$ is left-strict in both variables (or vice versa),
that is essentionally weaker condition than the strictness of the means.

In section 3 we give conditions under which the uniqueness of the invariant
mean guarantee the relevant convergence of the sequence of iterates of the
mean-type mapping (see Theorem~\ref{thm:unique}). If $K,M,N:I^{2}\rightarrow I$ are
arbitrary means and $K$ is a unique $\left( M,N\right) $-invariant mean,
then the sequence of iterates $\left( \left( M,N\right) ^{n}\right) _{n\in \N} $ of the mean-type mapping $\left( M,N\right) $ converges to $\left(
K,K\right) $ pointwise in $I^{2}.$)

In section 4 we improve the above result, namely, we show that the
result remains true on replacing the diagonal-contractivity of mean-type
mapping, by demanding only that, for every point outside some iterate of the
mean-type mapping is diagonally contractive, more precisely, that \textit{%
for every point }$\left( x,y\right) \in I^{2}\,,$ $x\neq y$\textit{, there
is a positive integer }$n=n\left( x,y\right) $ \textit{such that } 
\begin{equation*}
\left\vert M_{n}\left( x,y\right) -N_{n}\left( x,y\right) \right\vert
<\left\vert x-y\right\vert ,\text{ }
\end{equation*}%
\textit{where }$\left( M_{n},N_{n}\right) :=\left( M,N\right) ^{n}.$ Hereafter, whenever a pair $(M,N)$ is stated, the sequence $(M_n,\,N_n)$ defined in this way is also given.

\section{Preliminaries}

Let $I\subset \R$ be an interval. Recall that a function $%
M:I^{2}\rightarrow \R$ is called a \textit{mean} in $I,$ if it is
internal, i.e. if

\begin{equation*}
\min \left( x,y\right) \leq M\left( x,y\right) \leq \max \left( x,y\right) 
\text{, \ \ \ \ \ }x,y\in I\text{.}
\end{equation*}

The mean $M$ is called: \textit{strict}, if these inequalities are sharp
unless $x=y$; \textit{symmetric} if $M\left( x,y\right) =M\left(
y,x\right) $ for all $x,y\in I$.

Given means $K,M,N:I^{2}\rightarrow I$ in an interval $I.$ The mean $K$ is
called \textit{invariant} with respect to the mean-type mapping $\left(
M,N\right) :I^{2}\rightarrow I^{2}$, briefly, $\left( M,N\right) $\textit{%
-invariant}, if $K\circ \left( M,N\right) =K$ (\cite{JM-AeqMath99})$.$

\begin{remark}
If $M,N:I^{2}\rightarrow I$ are two means in an interval $I$ then, for all $n\in \N$, the mapping $\left(
M_{n},N_{n}\right)$ is
a mean-type mapping. Moreover every $(M,N)$-invariant mean is $(M_n,N_n)$-invariant for all $n \in\N$.
\end{remark}

\begin{remark}
(\cite{JM-AeqMath99}) If $K:I^{2}\rightarrow I$ is a continuous symmetric
and strictly increasing mean, then for every mean $M:I^{2}\rightarrow I$
there exists a unique function $N:I^{2}\rightarrow I$ such that $K\circ
\left( M,N\right) =K,$ and $N$ is a mean in $I$. Moreover, the mean $N$ is
called complementary to $M$ with respect to $K$.
\end{remark}

The following result is a consequence of bivariable version of Corollary
1 in \cite{JM2013} (see also \cite{BorBor}, \cite{JM99}, \cite{JM-ECIT2009}).

\begin{theorem}
If $M,N:I^{2}\rightarrow I$ are continuous means such that the mean-type
mapping $\left( M,N\right) $ is diagonally-contractive, i.e.%
\begin{equation} 
\left\vert M\left( x,y\right) -N\left( x,y\right) \right\vert <\left\vert
x-y\right\vert \text{, \ \ \ \ }x,y\in I\text{, }x\neq y\text{,}  \label{1}
\end{equation}%
then there is a unique mean $K: I^{2}\rightarrow I$ which is $\left(
M,N\right) $-invariant and, moreover, the sequence of iterates $\left( \left(
M,N\right) ^{n}\right) _{n\in N}$ of the mean-type mapping $\left(
M,N\right) $ converges (uniformly on compact subsets of $I^{2})$ to $\left(
K,K\right) $.
\end{theorem}

To formulate some simple sufficient conditions guarantying \eqref{1} let us introduce the following notions:

\begin{definition}
(\cite{JMforAczel}, \cite{ToaderCostin}) Let $M:I^{2}\rightarrow \R$
be a mean in an interval $I\subset \R$ and let $x_{0},y_{0}\in I$.
The mean $M:I^{2}\rightarrow \R$ is called:
\end{definition}

\qquad \textit{left-strict in the first variable at the point }$x_{0}$, if
for all $t\in I$,

\begin{equation*}
t<x_{0}\Longrightarrow M\left( x_{0},t\right) <x_{0}\:;
\end{equation*}%
\qquad \qquad \qquad

\qquad \textit{right-strict in the first variable at the point }$x_{0}$, if
for all $t\in I$,

\begin{equation*}
x_{0}<t\Longrightarrow x_{0}<M\left( x_{0},t\right)\:;
\end{equation*}
\qquad \qquad

\qquad \textit{left-strict in the second variable at the point }$y_{0}$, if
for all $t\in I$,

\begin{equation*}
t<y_{0}\Longrightarrow M\left( t,y_{0}\right) <y_{0}\:;
\end{equation*}%
\qquad \qquad \qquad

\qquad \textit{right-strict in the second variable at the point }$y_{0}$, if
for all $t\in I$,

\begin{equation*}
y_{0}<t\Longrightarrow y_{0}<M\left( t,y_{0}\right)\:.
\end{equation*}%
The mean $M$ is called \textit{left-strict in the first variable\ in a set }$%
A\subset I$, if it is left-strict in the first variable at every point $x_{0}\in A$; and $M$ is
called \textit{left-strict in the first variable}, if it is left-strict in the first variable at
every point. (Analogously one can introduce the remaining notions).

We adapt the convention that if a mean is left-strict (right-strict) in both variables then we call it simply \emph{left-strict} (\emph{right-strict}). 
Then 
\begin{align*}
 M \text{ is left-strict if and only if } \:\: 
 \big(x \ne y \Longrightarrow M(x,y) &< \max(x,y)\big);\\
  M \text{ is right-strict if and only if } \:\: 
  \big(x \ne y \Longrightarrow M(x,y) &> \min(x,y) \big).
\end{align*}

\begin{example}
The projective mean $P_{1}\colon \R^{2} \ni (x,y) \mapsto x$, is left and right-strict in the second
variable, but it is neither left nor right strict at any point in the first
variable. Similarly, the projective mean $P_2\colon \R^{2} \ni (x,y) \mapsto y$, is left and right-strict in
the first variable, but it is neither left nor right strict at any point in the
second variable.

The mean-type mapping $\left( P_{1},P_{2}\right) $ coincides with the
identity of $\R^{2}$, and of course, the sequence of its iterates $%
\left( \left( P_{1},P_{2}\right) ^{n}:n\in \N\right) $ converges to $%
\left( P_{1},P_{2}\right) $. Note that here $P_{1}\neq P_{2}$, moreover,
every mean $M:\R^{2}\rightarrow \R$ is $\left(
P_{1},P_{2}\right) $-invariant.
\end{example}

\begin{example}
Consider the extreme mean $\min :\R^{2}\mathbb{\rightarrow R}\,$.
Since, $t<x_{0}\Longrightarrow \min \left( x_{0},t\right) <x_{0}$ for every $%
t;$ the mean $\min $ is left-strict in the first variable and in the second
variable, but it is not right-strict in the first and the second variable.
Moreover, the extreme mean $\max :\R^{2}\mathbb{\rightarrow R}$ is
right-strict in the first and the second variable, but it is not left-strict
in the first and the second variable.

For every $n\in \N$, the $n$th iterate $\left( \min ,\max \right)
^{n}$ of mean-type mapping $\left( \min ,\max \right) $ has the form%
\begin{equation*}
\left( \min ,\max \right) ^{n}\left( x,y\right) =\left\{ 
\begin{array}{ccc}
\left( x,y\right) & if & x\leq y \\ 
\left( y,x\right) & if & x>y%
\end{array}%
\right.
\end{equation*}%
and, of course, the sequence of iterates of $\left( \left( \min ,\max
\right) ^{n}:n\in \N\right) $ converges. 

On the other hand, one can check that a mean is invariant with respect to the mean-type mapping $\left( \min ,\max \right)$ if and only if it is symmetric.
\end{example}

\begin{remark}
If $M:I^{2}\rightarrow \R$ is a strict mean then it is left and
right-strict in each of the two variables.
\end{remark}

Using the definition of one-side strict means we get the following result

\begin{proposition}
Let $M,N:I^{2}\rightarrow I$ be two means in an interval $I$%
. If one of the following conditions is valid:
\begin{enumerate}[(i)]
\item $M$ and $N$ are left-strict,
\item $M$ and $N$ are right-strict,
 \end{enumerate}
then property \eqref{1} holds.
\end{proposition}
\begin{proof}
Suppose to the contrary that there exists $x,\,y \in I$, $x \ne y$ such that 
$$|M(x,y)-N(x,y)|\ge|x-y|=|\max(x,y)-\min(x,y)|\:.$$
Binding this inequality with
$$\min(x,y)\le M(x,y)\le\max(x,y) \text{ and } \min(x,y)\le N(x,y)\le\max(x,y)$$
we obtain that $M(x,y)$ and $N(x,y)$ are two different endpoints of the interval $[\min(x,y),\,\max(x,y)]$. Thus $\{M(x,y),N(x,y)\}=\{\min(x,y),\:\max(x,y)\}$.

On the other hand, if $M$ and $N$ are both left-strict then none of them can be equal to $\max(x,y)$. Similarly, if $M$ and $N$ are both right-strict then none of them can be equal to $\min(x,y)$. These contradictions end the proof. 
\end{proof}

\subsection{Weakly contractive mean-type mappings}
In this section we deal with some relaxation of the assumption \eqref{A}. For two means $M,\ N \colon I^2 \to I$, the mean-type mapping $(M,N)$ is \emph{weakly contractive} if for
every elements $x,\,y \in I$ with $x\neq y$ there is a positive integer $n=n\left( x,y\right)$ such that
\begin{equation}
\left\vert M_{n}\left( x,y\right) -N_{n}\left( x,y\right) \right\vert
<\left\vert x-y\right\vert\:. \label{E:sh}
\end{equation}%

Next lemma provide a complete characterization of weak-contractivity in terms of $(M_2,N_2)$.
\begin{lemma}
Let $M,\ N \colon I^2 \to I$ be means.
Mean-type mapping $(M,N)$ is weakly contractive if and only if mean-type mapping $(M_2,N_2)$ is contractive.
\end{lemma}
\begin{proof}
We prove that if the inequality \eqref{E:sh} is satisfied for some triple $(x,y,n)$ then it is also valid for $(x,y,2)$. As $M$ and $N$ are means we get 
$$|M(x,y)-N(x,y)| \le |x-y|.$$
Therefore 
$$ \text{either}\quad
|M(x,y)-N(x,y)| < |x-y|\quad \text{ or }\quad |M(x,y)-N(x,y)|=|x-y|\:.
$$
In the first case we obtain validity of \eqref{E:sh} for the triple $(x,y,1)$ which, by mean property, implies its validity for $(x,y,2)$.

In the second case we get, by mean property, 
	$\{M(x,y),N(x,y)\}=\{x,y\}$. If $(M,N)(x,y)=(x,y)$ then $(M,N)^n(x,y)=(x,y)$ for all $n \in \N$ contradicting the assumption. Thus we obtain $M(x,y)=y$ and $N(x,y)=x$. Applying the same argumentation to the pair $(y,x)$ we get that either
$$|M_2(x,y)-N_2(x,y)| =|M(y,x)-N(y,x)|<|x-y|,$$
which implies \eqref{E:sh} for the triple $(x,y,2)$
or 
$$(M_2,N_2)(x,y)=(M,N)(y,x)=(x,y).$$
However in the second subcase we get 
$$(M_n,N_n)(x,y)=\begin{cases} (y,x) & \text{ if } n \text{ is odd,}\\
(x,y) & \text{ if } n \text{ is even,}
\end{cases}
$$
which implies that \eqref{E:sh} is valid for no $n \in \N$ contradicting the assumption. 

As the converse implication is trivial, the proof is complete
\end{proof}

\begin{proposition}
Let $M,N:I^{2}\rightarrow I$ be two means in an interval $I$%
. If the following conditions are valid:
\begin{enumerate}[(i)]
 \item $M$ is right-strict in the first variable or $N$ is left-strict in the second variable,
\item $M$ is left-strict in the second variable or $M$ is right-strict in the second variable or $N$ is left-strict in the first variable or $N$ is right-strict in the first variable,
\item $M$ is left-strict in the first variable or $N$ right-strict in the second variable,
 \end{enumerate}
then the mapping $(M,N)$ is diagonally-contractive.
\end{proposition}
\begin{proof}
 Fix $x,\,y \in I$, $x \ne y$. It suffices to prove that 
$$\big\{(x,y),(y,x)\big\} \not\subset \big\{(M,N)^n(x,y)\colon n \in \mathbb{N}\big\}\:.$$
First assume that $x<y$. Applying property (i) we obtain that $(M,N)(x,y) \ne (x,y)$. In view of property (iii) we get $(M,N)(y,x) \ne (y,x)$. 

The only remaining case is that $(M,N)(x,y)=(y,x)$ and $(M,N)(y,x)=(x,y)$. Equivalently,
$$M(x,y)=y,\quad N(x,y)=x,\quad M(y,x)=x,\quad \text{ and }\quad N(x,y)=y\:.$$
However these equalities cannot be simultaneously valid as they are excluded by consecutive alternatives in (ii), respectively.
The case $x>y$ is completely analogous.
\end{proof}

\bigskip

\bigskip
\section{Uniqueness of invariant means and convergence of iterates of
mean-type mappings}
\begin{definition}
Let $M,N:I^{2}\rightarrow I$ be means in an interval $I$. Define the \emph{diagonal basin} $\mathcal{A}_{M,N} \subset I^2$ by
$$
\Big\{(x,\,y) \in I^2 \colon \big((M_n,N_n)(x,\,y)\big)_{n=1}^\infty \text{ converges to some point on the diagonal } \Big\}.
$$
\end{definition}
\begin{theorem}
 Let $M,N:I^{2}\rightarrow I$ be means in an interval $I$. Then $\mathcal{A}_{M,N}$ is a maximal subset of $I^2$ such that all $(M,N)$-invariant means are equal to each other.
\end{theorem}
\begin{proof}
For $x,y \in I$. Let $a$ be a sequence which is a shuffling of $M_n$ and $N_n$, i.e
$$
a=(x,y,M_1(x,y),N_1(x,y),M_2(x,y),N_2(x,y),M_3(x,y),N_3(x,y),\dots)\:.
$$
By \cite{P.Pasteczka} we obtain that
$$L(x,y):=\liminf_{n \to \infty} a_n, \qquad 
U(x,y):=\limsup_{n \to \infty} a_n $$
are the smallest and the biggest $(M,\,N)$-invariant means, respectively. 

It is necessarily and also sufficient to prove that $L(x,y)=U(x,y)$ if and only if $(x,y)\in \mathcal{A}_{M,N}$.

If $L(x,y)=U(x,y)$ then we obtain that the sequence $a$ converges, whence $(M_n(x,y))_{n \in \N}$ and $(N_n(x,y))_{n \in \N}$ converge to a common limit, i.e. $(x,y) \in\mathcal{A}_{M,N}$.

Conversely, if $(x,y) \in \mathcal{A}_{M,N}$ then we get 
$$\lim_{n \to \infty} M_n(x,y)=\lim_{n \to \infty} N_n(x,y),$$
which implies that the sequence $a$ is convergent, and therefore $L(x,y)=U(x,y)$. As $L$ and $U$ is the smallest and the greatest $(M,N)$-invariant mean, we have that all $(M,N)$-invariant means have the same value at the point $(x,y)$.
\end{proof}

\begin{theorem} \label{thm:unique}
Let $K,M,N:I^{2}\rightarrow I$ be
arbitrary means.
$K$ is a unique $\left( M,N\right)$-invariant mean if and only if the sequence of iterates $\left( \left( M,N\right) ^{n}\right) _{n\in
N} $ of the mean-type mapping $\left( M,N\right) $ converges to $\left(
K,K\right) $ pointwise in $I^{2}.$\end{theorem}
\begin{proof}
For $x,y \in I$. Let $a$ be a sequence which is a shuffling of $M_n$ and $N_n$, i.e
$$
a=(x,y,M_1(x,y),N_1(x,y),M_2(x,y),N_2(x,y),M_3(x,y),N_3(x,y),\dots)\:.
$$
By \cite{P.Pasteczka} we obtain that
$$L(x,y):=\liminf_{n \to \infty} a_n, \qquad 
U(x,y):=\limsup_{n \to \infty} a_n $$
are the smallest and the biggest $(M,\,N)$-invariant means, respectively. 

$(\Rightarrow)$ As $K$ is a unique $(M,N)$-invariant mean we get $K=L=U$. It implies that the sequence $(a_n)$ is convergent to $K(x,y)$. Therefore both $M_n$ and $N_n$ converge to $K$ pointwise on $I^2$. In particular its Cartesian product $(M,N)^n=(M_n,N_n)$ converges to $(K,K)$ pointwise on the same set.

$(\Leftarrow)$ We know that the sequence $(a_n)$ is convergent to $K(x,y)$, in particular $K(x,y)=L(x,y)=U(x,y)$. As $x,\,y$ were taken arbitrarily we have $K=L=U$. 
It implies that the smallest and the biggest $(M,\,N)$-invariant mean coincide providing the uniqueness.
\end{proof}

\section{Mean-type mappings with diagonally-contractive iterates and
invariant means}

\begin{theorem}
If $M,N:I^{2}\rightarrow I$ are continuous means such that $(M,N)$ is weakly contractive then there exists a unique mean $K:I^{2}\rightarrow I$ which is $\left( M,N\right) $-invariant and, moreover, the sequence of iterates $\left( \left( M,N\right)
^{n}\right) _{n\in \N}$ of the mean-type mapping $\left( M,N\right) $
converges (uniformly on compact subsets of $I^{2})$ to $\left( K,K\right) .$
\end{theorem}

\begin{proof}
In view of Theorem~\ref{thm:unique} we get that $(M_2,N_2)$ is a contractive mean-type mapping, which is also continuous. Applying the result from \cite{JM2013} (see the beginning of this paper) we obtain that there exists a unique $(M_2,N_2)$-invariant mean. As every $(M,N)$-invariant mean is also $(M_2,N_2)$-invariant we get that there exists at most one $(M,N)$-invariant mean. 

On the other hand, applying some general results from \cite{P.Pasteczka}, we have that (under no assumptions for means) there exists at least one $(M,N)$-invariant mean. Therefore there exists a unique $(M,N)$ invariant mean; say $K \colon I^2 \to I$.

We can now apply Theorem~\ref{thm:unique} to obtain that $((M,N)^n)_{n \in \N}$ converges to $(K,K)$ pointwise on $I^2$. Moreover, in view of \cite[Corollary 4.4 (ii)]{JM2013}, this convergence is uniform on every compact subset of $I^2$.
\end{proof}

\begin{remark}
Let us emphasize that the continuity assumption in the previous theorem cannot be omitted; cf. \cite[section 3.1]{P.Pasteczka} for details.
\end{remark}

\begin{theorem}
Let $M,N:I^{2}\rightarrow I$ be (not necessarily continuous) means and $c\in %
\left[ 0,1\right) $. If \textit{for every point }$\left( x,y\right) \in
I^{2}\,,$ $x\neq y$\textit{, there is a positive integer }$n=n\left(
x,y\right) $ \textit{such that } 
\begin{equation*}
\left\vert M_{n}\left( x,y\right) -N_{n}\left( x,y\right) \right\vert
<c\left\vert x-y\right\vert ,\text{ }
\end{equation*}%
\textit{where }$\left( M_{n},N_{n}\right) :=\left( M,N\right) ^{n},$ then
there is a unique mean $K:I^{2}\rightarrow I$ that is $\left( M,N\right) $%
-invariant and, moreover the sequence of iterates $\left( \left( M,N\right)
^{n}\right) _{n\in N}$ of the mean-type mapping $\left( M,N\right) $
converges (uniformly on compact subsets of $I^{2})$ to $\left( K,K\right) .$
\end{theorem}

\begin{proof}
Take an arbitrary $\left( x,y\right) \in I^{2}$. Since $M$ and $N$ are means
in $I,$ for every $n\in \N$, we have%
\begin{equation*}
\min \left( x,y\right) \leq M\left( x,y\right) \leq \max \left( x,y\right) 
\text{, \ \ \ \ \ }\min \left( x,y\right) \leq N\left( x,y\right) \leq \max
\left( x,y\right) \text{,}
\end{equation*}%
hence%
\begin{align*}
\min \left( x,y\right) &\leq \min \left( M\left( x,y\right) ,N\left(
x,y\right) \right) \\
&\leq M\left( M\left( x,y\right) ,N\left( x,y\right)
\right) \leq \max \left( M\left( x,y\right) ,N\left( x,y\right) \right) \leq
\max \left( x,y\right) \text{,}
\end{align*}%
and%
\begin{align*}
\min \left( x,y\right) &\leq \min \left( M\left( x,y\right) ,N\left(
x,y\right) \right)\\
&\leq N\left( M\left( x,y\right) ,N\left( x,y\right)
\right) \leq \max \left( M\left( x,y\right) ,N\left( x,y\right) \right) \leq
\max \left( x,y\right) \text{,}
\end{align*}%
whence%
\begin{align*}
&\min \left( x,y\right) \leq\min \left( M_{1}\left( x,y\right)
,N_{1}\left( x,y\right) \right) \leq \min \left( M_{2}\left( x,y\right)
,N_{2}\left( x,y\right) \right)  \\
&\quad \leq \max \left( M_{2}\left( x,y\right) ,N_{2}\left( x,y\right) \right)
\leq \max \left( M_{1}\left( x,y\right) ,N_{1}\left( x,y\right) \right) \leq
\max \left( x,y\right) ,
\end{align*}%
and, by induction, for every $n\in \N$,%
\begin{align*}
&\min \left( x,y\right)  \leq \min \left( M_{n}\left( x,y\right)
,N_{n}\left( x,y\right) \right) \leq \min \left( M_{n+1}\left( x,y\right)
,N_{n+1}\left( x,y\right) \right)  \\
&\quad \leq \max \left( M_{n+1}\left( x,y\right) ,N_{n+1}\left( x,y\right)
\right) \leq \max \left( M_{n}\left( x,y\right) ,N_{n}\left( x,y\right)
\right) \leq \max \left( x,y\right) .
\end{align*}%
Thus the sequence $\left( \min \left( M_{n}\left( x,y\right) ,N_{n}\left(
x,y\right) \right) :n\in \N\right) $ is increasing and the sequence $%
\left( \max \left( M_{n}\left( x,y\right) ,N_{n}\left( x,y\right) \right)
:n\in \N\right)$ is decreasing; and both are convergent. Clearly, if 
$x=y$ or $M_{k}\left( x,y\right) =N_{k}\left( x,y\right) $ for some positive
integer $k$, then $M_{n}\left( x,y\right) =N_{n}\left( x,y\right) $ for all $%
n\geq k$, and consequently we have 
\begin{equation*}
\lim_{n\rightarrow \infty }M_{n}\left( x,y\right) =\lim_{n\rightarrow \infty
}N_{n}\left( x,y\right) .
\end{equation*}%
Assume that 
\begin{equation*}
\lim_{n\rightarrow \infty }\min \left( M_{n}\left( x,y\right) ,N_{n}\left(
x,y\right) \right) <\lim_{n\rightarrow \infty }\max \left( M_{n}\left(
x,y\right) ,N_{n}\left( x,y\right) \right) .
\end{equation*}%
In particular, we have $M_{n}\left( x,y\right) \neq N_{n}\left( x,y\right) $
for every $n\in \N$. By the assumption, for every $n\in \N$
there is $k=k_{n}\in \N$ such that 
\begin{align*}
\vert M_{k_{n}}\left( M_{n}\left( x,y\right) ,N_{n}\left( x,y\right)
\right) &- N_{k_{n}}\left( M_{n}\left( x,y\right) ,N_{n}\left( x,y\right)
\right) \vert \\
&\qquad<c\left\vert M_{n}\left( x,y\right) -N_{n}\left(
x,y\right) \right\vert 
\end{align*}%
that is that%
\begin{equation*}
\left\vert M_{n+k_{n}}\left( x,y\right) -N_{n+k_{n}}\left( x,y\right)
\right\vert <c\left\vert M_{n}\left( x,y\right) -N_{n}\left( x,y\right)
\right\vert ,
\end{equation*}%
which leads to a contradiction. Thus%
\begin{equation*}
\lim_{n\rightarrow \infty }\min \left( M_{n}\left( x,y\right) ,N_{n}\left(
x,y\right) \right) =\lim_{n\rightarrow \infty }\max \left( M_{n}\left(
x,y\right) ,N_{n}\left( x,y\right) \right) ,
\end{equation*}%
and consequently, 
\begin{equation*}
\lim_{n\rightarrow \infty }M_{n}\left( x,y\right) =\lim_{n\rightarrow \infty
}N_{n}\left( x,y\right) .
\end{equation*}%
Finally, let $K$ be a pointwise limit of $(M_n)$ or $(N_n)$, as they are equal.
\end{proof}

\end{document}